\documentclass[12pt]{article}

\usepackage[a4paper,margin=1in]{geometry}
\usepackage{amsmath,amssymb,amsfonts,amsthm,mathtools}
\usepackage{mathrsfs}
\usepackage{hyperref}
\usepackage{booktabs,multirow,array,longtable}
\usepackage{graphicx}
\usepackage{tikz,cite}
\usepackage{algorithm}
\usepackage{algpseudocode}
\usepackage{enumitem}
\usepackage{xcolor}
\usetikzlibrary{arrows.meta,positioning,shapes.geometric,calc,fit}
\everymath{\displaystyle}
\renewenvironment{proof}{{\bf \noindent Proof.}}{\qed}

\hypersetup{
	colorlinks=true,
	linkcolor=blue!60!black,
	citecolor=blue!60!black,
	urlcolor=blue!60!black
}

\newtheorem{theorem}{Theorem}[section]
\newtheorem{lemma}[theorem]{Lemma}
\newtheorem{proposition}[theorem]{Proposition}
\newtheorem{corollary}[theorem]{Corollary}
\theoremstyle{definition}

\newtheorem{example}[theorem]{Example}
\newtheorem{remark}[theorem]{Remark}
\newtheorem{problem}[theorem]{Problem}

\newcommand{\F}{\mathbb F}

\title{Vertex connectivity of the nonzero nonunit core of the comaximal graph of $\mathbb Z_n$.}

\author{Bilal Ahmad Rather\\[2mm]
	\small School of Mathematics and Statistics, Shandong University of Technology,\\
	\small Zibo 255049, China\\
	\texttt{\href{mailto:bilalahmadrr@gmail.com}{bilalahmadrr@gmail.com}}
}

\date{}

\begin{document}
	
	\maketitle
	\pagestyle{myheadings}
	\markboth{Bilal Ahmad Rather}{Exact vertex connectivity of $G_2$}
	
	\begin{abstract}
		This article settles Problem 7.2 posed by [Banerjee, Special Matrices (2022)] for the induced subgraph $G_2$ of the comaximal graph $\Gamma(\mathbb Z_n)$ when $n$ is squarefree. Let $n=p_1p_2\cdots p_m$ with distinct primes $p_1<\cdots<p_m$, and let $G_2$ be the graph on the nonzero nonunit residue classes modulo $n$. We use Chinese remainder representation of $\mathbb Z_n$, and encodes each vertex by the set of vanishing coordinates. This converts $G_2$ into a weighted blow-up of a disjointness graph on nonempty proper subsets of $\{1,\dots,m\}$. Within this model, we derive exact class sizes, explicit degree formulas, the minimum-degree layer, and a short-path criterion. The main theorem proves the connectivity of $G_{2}$ as  $\kappa(G_2)=\prod_{i=1}^{m-1}(p_i-1)=\tfrac{\phi(n)}{p_m-1}$. Consequently, earlier upper bound is sharp, $G_2$ is maximally connected, and its edge connectivity agrees with its minimum degree. We also obtain distance formulas, diameter and radius information, and a linear-time algorithm once the prime factorization is known.
	\end{abstract}
	
	
	
	\section{Introduction}\label{sec:intro}
	
	Graphs built from algebraic objects frequently expose structural information that is hard to read directly from the original algebra. Among the best known examples are zero-divisor graphs, unit graphs, annihilating-ideal graphs, and comaximal graphs. In the present setting, the relevant object is the comaximal graph $\Gamma(R)$ of a commutative ring $R$ with identity, introduced by Sharma and Bhatwadekar \cite{SharmaBhatwadekar1995}. Its vertices are the elements of $R$, and two distinct elements are adjacent exactly when they generate the whole ring as an ideal. Even this basic definition already mixes algebra and graph theory in a natural way, ideal generation becomes adjacency, maximal ideals manifest themselves as obstruction sets, and ring decomposition often turns into graph decomposition. The finite residue ring $\mathbb Z_n$ is especially attractive because it is simultaneously elementary, canonical, and rich enough to display nontrivial spectral and connectivity phenomena.
	
	The graph $\Gamma(\mathbb Z_n)$ has been investigated from several viewpoints. Structural aspects of comaximal graphs for general commutative rings were studied by Maimani, Salimi, Sattari, and Yassemi \cite{Maimani2008}, by Moconja and Petrović \cite{MoconjaPetrovic2011}, and by Samei \cite{Samei2014}. These works clarified how algebraic decompositions of the ring govern graph-theoretic properties such as connectedness, domination, and local neighborhoods. On the graph-theoretic side, the study of connectivity and Laplacian invariants belongs to the broader program of understanding how robust a network is under vertex failures, for that background we refer to West \cite{West2001}. The spectral language initiated by Fiedler \cite{Fiedler1973} and surveyed by Merris \cite{Merris1994}, Mohar \cite{Mohar1991}, Brouwer and Haemers \cite{BrouwerHaemers2012}, and de Abreu \cite{deAbreu2007} is especially relevant because algebraic connectivity and vertex connectivity often interact in subtle ways. Other algebraic invariants are given in \cite{bilalcumn,bilaliaaecc,bilaliac,bilalijaa,bilalijpam}.
	
	A decisive step for $\Gamma(\mathbb Z_n)$ was made in \cite{Banerjee2022,afkhami}, they analyzed the Laplacian spectrum and several structural features of the graph. The induced subgraph on the nonzero nonunit elements,
	$$
	G_2=\Gamma(\mathbb Z_n)\big[\{x\in \mathbb Z_n:\ x\neq 0,\ \gcd(x,n)\neq 1\}\big],
	$$
	appears as the genuinely nontrivial core of the full comaximal graph. Indeed, units are universal vertices, while the zero element has a highly controlled neighborhood. An equitable partition of $G_2$ can be obtained based on proper divisor $d_{i}$. Its structure is described through divisor classes (see remark \ref{remark}), several spectral statements are proved, and  an upper bound if established for the vertex connectivity of $G_2$ when $n$ is squarefree \cite{Banerjee2022,afkhami}. The paper \cite{Banerjee2022} then ended with a natural open problem: determine the exact value of $\kappa(G_2)$ for
	$$
	n=p_1p_2\cdots p_m,\qquad 2\leq p_1<p_2<\cdots<p_m,
	$$
	where the $p_i$ are prime. This is Problem 7.2 in \cite{Banerjee2022}.
	
	The importance of this question is both technical and conceptual. Technically, vertex connectivity measures the smallest number of vertices whose deletion disconnects the graph, and hence it is the most basic fault-tolerance parameter. In the arithmetic setting of $\Gamma(\mathbb Z_n)$, a minimum vertex cut identifies the thinnest layer through which the nontrivial comaximal interactions must pass. Conceptually, an exact formula for $\kappa(G_2)$ reveals which prime factors control the fragile directions of the graph. At first sight one might expect the answer to depend on the whole factorization of $n$ in a complicated way, because the adjacency relation on $G_2$ is defined globally through comaximality. One of the main messages of the present paper is that, after the right change of viewpoint, the answer becomes remarkably clean.
	
	Our method is based on the Chinese remainder theorem. When $n$ is squarefree, the ring $\mathbb Z_n$ is isomorphic to the direct product $\F_{p_1}\times\cdots\times \F_{p_m}$, and every residue class can be written as a coordinate vector. For vertices of $G_2$, at least one coordinate is zero and at least one coordinate is nonzero. We therefore encode a vertex by its zero-set, namely the subset of coordinates on which the vector vanishes. This transforms the ring-theoretic adjacency relation into an elementary set-theoretic one: two vertices are adjacent if and only if their zero-sets are disjoint. In other words, $G_2$ becomes a weighted blow-up of a disjointness graph on the nonempty proper subsets of $[m]=\{1,2,\dots,m\}$. Once this model is available, the exact connectivity problem can be attacked by purely combinatorial means. 
	
	The first outcome of this approach is a transparent decomposition of the vertex set into layers $X_S$, indexed by the nonempty proper subsets $S\subsetneq [m]$. Each layer $X_S$ consists of all vertices whose zero-set is $S$, and its cardinality is given explicitly by
	$ |X_S|=\prod_{i\notin S}(p_i-1). $
	The second outcome is an exact degree formula. For a vertex in $X_S$, the degree depends only on $S$, and the minimum degree occurs precisely on the layer whose zero-set omits the largest prime index. The third and decisive step is to show that the layer $X_{{m}}$ is a separator of size $\prod_{i=1}^{m-1}(p_i-1)$ and that removing fewer vertices can never disconnect the graph. This proves the exact formula
	$$
	\kappa(G_2)=\prod_{i=1}^{m-1}(p_i-1)=\frac{\phi(n)}{p_m-1}.
	$$
	Thus the upper bound from \cite{Banerjee2022} is sharp, and the graph $G_2$ turns out to be maximally connected in the classical sense that $\kappa(G_2)=\delta(G_2)$.
	
	Beyond the open problem itself, the support-set model yields several additional consequences. We obtain a complete distance formula between two vertices in terms of their zero-sets, and this leads to a short proof that the diameter is $3$ whenever $m\geq 3$. We also identify the central layers of the graph and provide a simple algorithm for computing $\kappa(G_2)$ once the distinct prime factors of $n$ are known. The algorithmic viewpoint is useful, as it separates two stages of the computation, prime factorization, which is external to the graph problem, and the graph-theoretic evaluation of the connectivity, which is linear in the number of prime factors. In this way the paper not only answers the open problem but also clarifies the mechanism behind the answer.
	
	From a literature perspective, the main novelty of the present work lies in replacing the divisor-based description of \cite{Banerjee2022,afkhami} by a support-set geometry tailored to the squarefree case. That reformulation seems to be absent from the existing comaximal-graph literature. It reveals that $G_2$ behaves like a nonuniform Kneser-type disjointness graph, with weights inherited from the residue sizes of the coordinate fields. This change of model simplifies the proofs, isolates the extremal cut, and naturally suggests further questions for nonsquarefree moduli, edge connectivity, and the classification of all minimum separators.
	
	\medskip
	
	The article is organized as: Section \ref{sec:prelim} fixes notation, records the known results that we use repeatedly, and explains the precise gap left open by previous work (in particular Problem 7.2 in \cite{Banerjee2022}). Section \ref{sec:model} introduces the support-set representation and proves that $G_2$ is a weighted blow-up of a disjointness graph. Section \ref{sec:degree} derives the neighborhood and degree formulas, identifies the unique minimum-degree layer, and prepares the connectivity argument. Section \ref{sec:connectivity} contains the exact solution to Problem 7.2 o f\cite{Banerjee2022}, proves that the bound is sharp, and deduces maximal connectivity. Section \ref{sec:distance} studies metric consequences of the model, including an exact distance formula, the diameter, and short routing statements. Section \ref{sec:algorithmic} gives an explicit algorithm, establishes its correctness and complexity, and collects numerical comparisons and computation tables. Finally, Section \ref{sec:conclusion} summarizes the contribution, discusses limitations, and records directions for further research.
	
	\section{Preliminaries and notation}\label{sec:prelim}
	
	Throughout the paper, all graphs are finite, simple, and undirected. For a graph $G$, the vertex set is denoted by $V(G)$, the neighborhood of a vertex $v$ by $N_G(v)$, the degree of $v$ by $\deg_G(v)$, the minimum degree by $\delta(G)$, the vertex connectivity by $\kappa(G)$, and the edge connectivity by $\lambda(G)$.  
	Two vertices $u\neq v$ are called \emph{false twins}, if they have exactly the same open neighborhood and are \emph{not} adjacent, that is,  $N(u) = N(v)$ and $uv \notin E(G).$	
	We fix a squarefree modulus
	$ n=p_1p_2\cdots p_m,$ with $ 2\leq p_1<p_2<\cdots<p_m, $
	where $m\geq 2$ and the $p_i$ are distinct primes. We write $[m]=\{1,2,\dots,m\}$. Let $\Gamma(\mathbb Z_n)$ be the comaximal graph of $\mathbb Z_n$, and let $G_2$ be the induced subgraph on the nonzero nonunit elements.
	 For a subset $S\subseteq [m]$, we write $ d_S=\prod_{i\in S}p_i, $ 	with the convention that the empty product equals $1$. For a vertex $x\in G_2$, its zero-set will later be denoted by $Z(x)\subseteq [m]$. 
	 For a squarefree modulus $n$, the support classes in $G_2$ will be indexed by nonempty proper subsets of $[m]$. 
	
	The following known facts form the background of our arguments. The next theorem is the arithmetic form of the Chinese remainder theorem used to encode vertices by coordinate zero patterns.	
	\begin{theorem}[Chinese remainder theorem, Chap.~7, \cite{DummitFoote2004}]\label{thm:CRT}
		There is a ring isomorphism
		$$
		\mathbb Z_n\cong \F_{p_1}\times \F_{p_2}\times\cdots\times \F_{p_m}.
		$$
		Under this identification, each element of $\mathbb Z_n$ may be viewed as an $m$-tuple $(x_1,\dots,x_m)$ with $x_i\in \F_{p_i}$.
	\end{theorem}
	
	\medskip
	 The following criterion translates comaximality in $\mathbb Z_n$ into a gcd condition, (Banerjee, eq.~(2) \cite{Banerjee2022}, see also \cite{afkhami}). For $x,y\in \mathbb Z_n$, the vertices $x$ and $y$ are adjacent in $\Gamma(\mathbb Z_n)$ if and only if
		\begin{equation}\label{prop:gcdcriterion}
		\gcd\big(\gcd(x,n),\gcd(y,n)\big)=1 \qquad\text{or}\qquad \gcd(x,y,n)=1.
		\end{equation} 
	
	\medskip
	 The next inequality is the standard connectivity chain that will allow us to identify edge connectivity once vertex connectivity is known.	
	\begin{proposition}[Whitney inequalities, Theorem~4.2.10, \cite{West2001}]\label{prop:Whitney}
		For every connected graph $G$,
		$$
		\kappa(G)\leq \lambda(G)\leq \delta(G).
		$$
	\end{proposition}
	
	\medskip
	 The following statement records the best available bound from the earlier literature and is the starting point of the present paper.	
	\begin{proposition}[Banerjee, Theorem~6.3, \cite{Banerjee2022}]\label{prop:BanerjeeUpper}
		Let $n=p_1p_2\cdots p_m$ be squarefree with distinct primes in increasing order. Then
		$$
		\kappa(G_2)\leq \phi(p_1p_2\cdots p_{m-1})=\prod_{i=1}^{m-1}(p_i-1).
		$$
	\end{proposition}
	
	\medskip
	 The next observation connects the ring-theoretic notation from \cite{Banerjee2022,afkhami} with the graph studied here.
	
	\begin{remark}\label{remark}
		In \cite{Banerjee2022,afkhami}, a graph denoted by $G_2$, is described through the divisor classes
		$$
		A_d=\{x\in \mathbb Z_n:\ \gcd(x,n)=d\},
		$$
		where $d$ runs over the proper divisors of $n$. The present paper keeps the notation $G_2$ for the same induced subgraph but reindexes the classes by subsets of $[m]$, which is more efficient in the squarefree case.
	\end{remark}
	
	 The existing literature  gives a strong structural and spectral description of $\Gamma(\mathbb Z_n)$ and, in particular, an upper bound for $\kappa(G_2)$ in the squarefree case \cite{Banerjee2022}. What was missing is an exact formula for $\kappa(G_2)$, a proof of sharpness of the bound $\kappa(G_2)\leq \phi(p_1p_2\cdots p_{m-1})$, and a combinatorial explanation of why the extremal separator is tied to the largest prime factor. Likewise, no explicit distance formula for $G_2$ seems to have been isolated in the squarefree setting. The remainder of the paper fills these gaps.
	
	\section{A weighted disjointness model for $G_2$}\label{sec:model}
	 The results of this section recast $G_2$ as a weighted disjointness graph. This viewpoint is not explicit in \cite{Banerjee2022} and is the combinatorial engine behind all later connectivity and distance proofs.
	
	Using Theorem \ref{thm:CRT}, we identify $\mathbb Z_n$ with $\F_{p_1}\times\cdots\times\F_{p_m}$.
	 For $x=(x_1,\dots,x_m)\in \mathbb Z_n$, define the zero-set
		$$
		Z(x)=\{i\in [m]: x_i=0\}.
		$$
		For every nonempty proper subset $S\subsetneq [m]$, define
		$$
		X_S=\{x\in V(G_2): Z(x)=S\}.
		$$ 
	
	\medskip
	 The next proposition identifies exactly which zero-sets occur in the graph $G_2$.
	\begin{proposition}\label{prop:nonemptyproper}
		A vector $x=(x_1,\dots,x_m)\in \mathbb Z_n$ belongs to $V(G_2)$ if and only if $Z(x)$ is a nonempty proper subset of $[m]$.
	\end{proposition}
	
	\begin{proof}
		A vertex of $G_2$ is by definition nonzero and nonunit. In the product ring $\F_{p_1}\times\cdots\times\F_{p_m}$, an element is a unit exactly when all coordinates are nonzero, and it is zero exactly when all coordinates are zero. Therefore, $x$ is nonzero and nonunit if and only if at least one coordinate vanishes and at least one coordinate is nonzero, which is equivalent to saying that $Z(x)$ is nonempty and proper.
	\end{proof}
	
	\medskip
	 The next proposition gives the exact size of each support class. 
	\begin{proposition}\label{prop:classsize}
		For every nonempty proper subset $S\subsetneq [m]$,
		$$
		|X_S|=\prod_{i\notin S}(p_i-1).
		$$
	\end{proposition}
	
	\begin{proof}
		A vector lies in $X_S$ exactly when its coordinates indexed by $S$ are zero and its coordinates indexed by $[m]\setminus S$ are nonzero. The coordinates in $S$ are forced, while for every $i\notin S$ there are exactly $p_i-1$ nonzero choices in $\F_{p_i}$. Multiplying the independent choices, we obtain
		$$
		|X_S|=\prod_{i\notin S}(p_i-1).
		$$
	\end{proof}
	
	\medskip
	 The next theorem converts ring-theoretic adjacency into set-theoretic disjointness.	
	\begin{theorem}\label{thm:adjdisjoint}
		Let $x\in X_S$ and $y\in X_T$, where $S$ and $T$ are nonempty proper subsets of $[m]$. Then $x$ and $y$ are adjacent in $G_2$ if and only if $S\cap T=\emptyset$.
	\end{theorem}
	
	\begin{proof}
		By Equation \ref{prop:gcdcriterion}, $x$ and $y$ are adjacent in $\Gamma(\mathbb Z_n)$ if and only if
		$$
		\gcd\big(\gcd(x,n),\gcd(y,n)\big)=1.
		$$
		Since $n$ is squarefree and $x\in X_S$, we have $\gcd(x,n)=d_S=\prod_{i\in S}p_i$. Likewise, $\gcd(y,n)=d_T=\prod_{i\in T}p_i$. Thus, we have
		$$
		\gcd\big(\gcd(x,n),\gcd(y,n)\big)=\gcd(d_S,d_T)=\prod_{i\in S\cap T}p_i.
		$$
		This equals $1$ precisely when $S\cap T=\emptyset$. Since $G_2$ is an induced subgraph of $\Gamma(\mathbb Z_n)$, and the same criterion describes adjacency inside $G_2$.
	\end{proof}
	
	\medskip
	 The following corollary identifies the global structure of $G_2$ as a weighted blow-up.	
	\begin{corollary}\label{cor:blowup}
		Let $\mathcal D_m$ be the graph whose vertex set is
		$$
		\mathscr P^\ast([m])=\{S\subseteq [m]:\ \emptyset\neq S\subsetneq [m]\},
		$$
		where two subsets are adjacent exactly when they are disjoint. Then $G_2$ is the blow-up of $\mathcal D_m$ obtained by replacing each vertex $S$ of $\mathcal D_m$ with the independent set $X_S$ of size $\prod_{i\notin S}(p_i-1)$.
	\end{corollary}
	
	\begin{proof}
		By Proposition \ref{prop:nonemptyproper}, the vertex set of $G_2$ is the disjoint union of the classes $X_S$ over all nonempty proper $S\subsetneq [m]$. By Theorem \ref{thm:adjdisjoint}, there are no edges inside a class $X_S$, and between two distinct classes $X_S$ and $X_T$ either all possible edges appear or none appear, according as $S\cap T=\emptyset$ or not. This is exactly the definition of the stated blow-up of $\mathcal D_m$.
	\end{proof}
	
	\medskip
	 The next corollary records the order of the graph and serves as a consistency check for the partition.	
	\begin{corollary}\label{cor:order}
		The graph $G_2$ has $ |V(G_2)|=n-1-\phi(n) $ vertices.
	\end{corollary}
	
	\begin{proof}
		A vertex of $G_2$ is precisely a residue class modulo $n$ that is neither zero nor a unit. There are $n$ residue classes in total, exactly one of them is zero, and exactly $\phi(n)$ of them are units. Hence, we obtain
		$$
		|V(G_2)|=n-1-\phi(n).
		$$
	\end{proof}
	
	The partition in terms of the gcd classes $A_d$ is already sufficient for spectral calculations \cite{Banerjee2022, afkhami}, but the support-set indexing used here is sharper for the squarefree case. It turns the arithmetic of gcd values into the combinatorics of disjoint subsets and makes the location of the extremal cut visible before any calculation begins.
	\begin{example}\label{ex:n30classes}
		Let $n=30=2\cdot 3\cdot 5$, so $m=3$. The nonempty proper subsets of $[3]$ are
		$$
		\{1\},\{2\},\{3\},\{1,2\},\{1,3\},\{2,3\}.
		$$
		By Proposition \ref{prop:classsize}, the corresponding class sizes are listed in Table \ref{tab:n30classes}. Their sum is $21=30-1-\phi(30)$, which is in agreement with Corollary \ref{cor:order}. The adjacency pattern is displayed in Figure \ref{fig:n30model}, where each node represents a class $X_S$, and the number below the class label is its cardinality. The figure shows that singleton classes are pairwise adjacent, while a two-element support class is adjacent only to its complementary singleton class.
	\end{example}
	\begin{table}[H]
		\centering
		\caption{Support classes for $n=30$.}
		\label{tab:n30classes}
		\begin{tabular}{ccc}
			\toprule
			Subset $S$ & Divisor $d_S$ & $|X_S|$ \\
			\midrule
			$\{1\}$ & $2$ & $8$ \\
			$\{2\}$ & $3$ & $4$ \\
			$\{3\}$ & $5$ & $2$ \\
			$\{1,2\}$ & $6$ & $4$ \\
			$\{1,3\}$ & $10$ & $2$ \\
			$\{2,3\}$ & $15$ & $1$ \\
			\bottomrule
		\end{tabular}
	\end{table}
	
	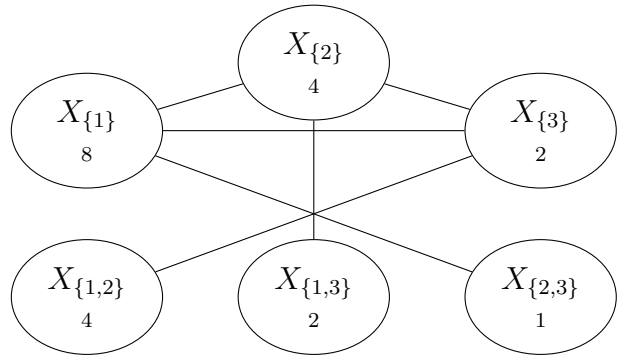
\begin{figure}[H]
		\centering
		\begin{tikzpicture}[
			every node/.style={draw, ellipse, align=center, minimum width=20mm, minimum height=10mm},
			scale=1
			]
			\node (a1) at (0,2.2) {$X_{\{1\}}$\\ $\scriptstyle 8$};
			\node (a2) at (3,3.1) {$X_{\{2\}}$\\ $\scriptstyle 4$};
			\node (a3) at (6,2.2) {$X_{\{3\}}$\\ $\scriptstyle 2$};
			\node (b12) at (0,0) {$X_{\{1,2\}}$\\ $\scriptstyle 4$};
			\node (b13) at (3,0) {$X_{\{1,3\}} $\\ $\scriptstyle 2$};
			\node (b23) at (6,0) {$X_{\{2,3\}}$\\ $\scriptstyle 1$};
			
			\draw (a1)--(a2);
			\draw (a1)--(a3);
			\draw (a2)--(a3);
			\draw (b12)--(a3);
			\draw (b13)--(a2);
			\draw (b23)--(a1);
		\end{tikzpicture}
		\vspace*{-2mm}
		\caption{Weighted disjointness model for $n=30$.}
		\label{fig:n30model}
	\end{figure}
	
	\section{Degree profile and local structure}\label{sec:degree}
	 This section derives explicit local formulas that are missing in the existing literature. In particular, we identify the unique minimum-degree layer and isolate the exact neighborhood that later becomes a minimum vertex cut.
	
	For $S\subseteq [m]$, we write $S^c=[m]\setminus S$. The next proposition identifies the entire neighborhood of a vertex from its zero-set alone.	
	\begin{proposition}\label{prop:neighborhood}
		If $x\in X_S$, then
		$$
		N(x)=\bigsqcup_{\emptyset\neq T\subseteq S^c} X_T.
		$$
		In particular, the neighborhood depends only on $S$.
	\end{proposition}
	
	\begin{proof}
		By Theorem \ref{thm:adjdisjoint}, a vertex $y\in X_T$ is adjacent to $x\in X_S$ if and only if $T\cap S=\emptyset$. Since every $T$ indexing a vertex of $G_2$ is nonempty and proper, the condition $T\cap S=\emptyset$ is equivalent to $\emptyset\neq T\subseteq S^c$. The union is disjoint, since the classes $X_T$ are disjoint.
	\end{proof}
	
	\medskip
	 The next theorem gives an exact degree formula in terms of the support set.	
	\begin{theorem}\label{thm:degreeformula}
		If $x\in X_S$,  then
		$$
		\deg(x)=\sum_{\emptyset\neq T\subseteq S^c}\prod_{i\notin T}(p_i-1)
		$$
		and, equivalently,
		$$
		\deg(x)=\prod_{i\in S}(p_i-1)\left(\prod_{j\in S^c}p_j-\prod_{j\in S^c}(p_j-1)\right).
		$$
	\end{theorem}
	
	\begin{proof}
		The first identity follows immediately from Proposition \ref{prop:neighborhood} and Proposition \ref{prop:classsize}, since the degree is the sum of the sizes of all neighboring classes. For the second identity, we count the neighbors directly in the coordinate model. Since $x\in X_S$, a vector $y=(y_1,\dots,y_m)$ is adjacent to $x$ if and only if the following two conditions hold:  (i) for every $i\in S$, the coordinate $y_i$ is nonzero, and (ii)   on the complementary set $S^c$, the vector $(y_j)_{j\in S^c}$ is not entirely nonzero, as $y$ must remain a nonunit. There are $\prod_{i\in S}(p_i-1)$ ways to choose the coordinates indexed by $S$. On the coordinates indexed by $S^c$, there are $\prod_{j\in S^c}p_j$ arbitrary choices and $\prod_{j\in S^c}(p_j-1)$ forbidden choices in which all coordinates are nonzero. Multiplying these independent choices, we obtain
		$$
		\deg(x)=\prod_{i\in S}(p_i-1)\left(\prod_{j\in S^c}p_j-\prod_{j\in S^c}(p_j-1)\right).
		$$
	\end{proof}
	
	Figure \ref{fig:degreeblock} shows the block diagram for the counting argument in Theorem \ref{thm:degreeformula}, where it represents the two independent parts of the neighbor count, forced nonzero coordinates on $S$ and partially constrained coordinates on $S^c$.
	\begin{figure}[ht]
		\centering
		\begin{tikzpicture}[
			box/.style={draw, rounded corners, minimum width=36mm, minimum height=10mm, align=center},
			arr/.style={-Latex, thick}
			]
			\node[box, fill=blue!8] (v) at (0,0) {$x\in X_S$};
			\node[box, fill=green!8, right=50mm of v] (s) {coordinates in $S$\\ must be nonzero};
			\node[box, fill=yellow!12, below=12mm of s] (c) {coordinates in $S^c$\\ arbitrary, but not all nonzero};
			\node[box, fill=red!8, below=12mm of v] (d) {$\deg(x)=\prod_{i\in S}(p_i-1)\left(\prod_{j\in S^c}p_j-\prod_{j\in S^c}(p_j-1)\right)$};
			
			\draw[arr] (v) -- (s);
			\draw[arr] (v) -- (d);
			\draw[arr] (s) -- (c);
			\draw[arr] (c) -- (d);
		\end{tikzpicture}
		\caption{Block diagram for the counting argument in Theorem \ref{thm:degreeformula}.}
		\label{fig:degreeblock}
	\end{figure}
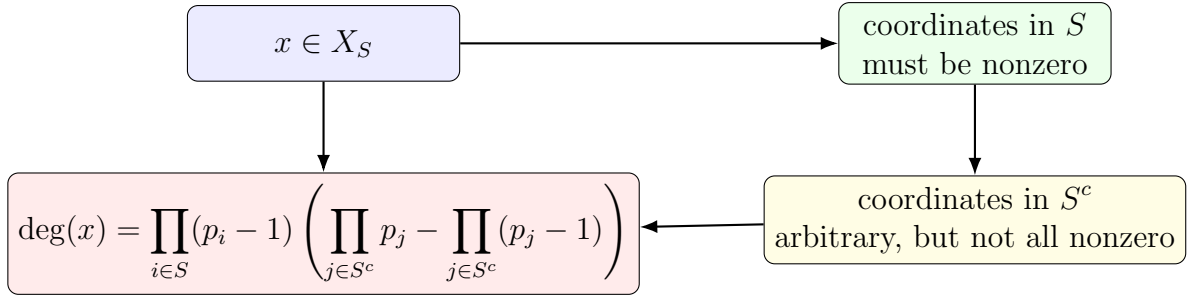
	\medskip
	The following corollary shows that each layer consists of pairwise nonadjacent vertices with identical open neighborhoods.	
	\begin{corollary}\label{cor:falsetwins}
		If $x,y\in X_S$, then $x$ and $y$ are nonadjacent and $N(x)=N(y)$. Thus, each class $X_S$ is a false-twin class.
	\end{corollary}
	
	\begin{proof}
		Since $S\cap S=S\neq\emptyset$, Theorem \ref{thm:adjdisjoint} shows that $x$ and $y$ are not adjacent. The equality of neighborhoods is exactly the second statement in Proposition \ref{prop:neighborhood}.
	\end{proof}
	
	\medskip
	 The next theorem locates the minimum degree exactly. 
	\begin{theorem}\label{thm:mindegree}
		The minimum degree of $G_2$ is
		$$
		\delta(G_2)=\prod_{i=1}^{m-1}(p_i-1),
		$$
		and this value is attained precisely by the vertices in the class
		$$
		X_{[m]\setminus\{m\}}.
		$$
	\end{theorem}
	
	\begin{proof}
		Let $x\in X_S$, and and let $J=S^c$. Then by Proposition \ref{prop:neighborhood}, we have
		$$
		N(x)=\bigsqcup_{\emptyset\neq T\subseteq J}X_T.
		$$
		Hence, $N(x)$ contains each singleton class $X_{\{j\}}$ with $j\in J$. Therefore, we obtain
		$$
		\deg(x)\geq \sum_{j\in J}|X_{{j}}|.
		$$
		Now, $ |X_{\{j\}}|=\prod_{i\neq j}(p_i-1). $ As $p_1<\cdots<p_m$, the smallest among these singleton-class sizes is
		$$
		|X_{\{m\}}|=\prod_{i=1}^{m-1}(p_i-1).
		$$
		 If $|J|\geq 2$, then $\deg(x)$ contains at least two disjoint singleton layers in its neighborhood, so we have
		$$
		\deg(x)\geq |X_{\{j_1\}}|+|X_{\{j_2\}}|>|X_{\{m\}}|
		$$
		for any distinct $j_1,j_2\in J$. Thus,  no such vertex has minimum degree. If $|J|=1$, say $J=\{j\}$, then $S=[m]\setminus\{j\}$. In this case, the only nonempty subset of $J$ is $\{j\}$ itself, so Proposition \ref{prop:neighborhood} implies that $N(x)=X_{\{j\}},$ 	and therefore, we have
		$$
		\deg(x)=|X_{\{j\}}|=\prod_{i\neq j}(p_i-1).
		$$
		This is minimized exactly when $j=m$, and hence
		$$
		\delta(G_2)=|X_{\{m\}}|=\prod_{i=1}^{m-1}(p_i-1),
		$$
		with equality holds precisely on the class $X_{[m]\setminus\{m\}}$.
	\end{proof}
	
	\medskip
	 The next corollary identifies the common neighborhood of every minimum-degree vertex.
	\begin{corollary}\label{cor:minlayerneighborhood}
		For every vertex $x\in X_{[m]\setminus\{m\}}$,
		$ N(x)=X_{\{m\}},$ and in particular, all minimum-degree vertices share the same neighborhood.
	\end{corollary}
	
	\begin{proof}
		This is the special case $S=[m]\setminus\{m\}$ of Proposition \ref{prop:neighborhood}, since then $S^c=\{m\}$ and the only nonempty subset of $S^c$ is $\{m\}$.
	\end{proof}
	
	 In \cite{Banerjee2022,afkhami}, the authors isolates a separator of cardinality $\phi(p_1\cdots p_{m-1})$, but the local reason for its extremality is not transparent there. Theorem \ref{thm:mindegree} and Corollary \ref{cor:minlayerneighborhood} show that this separator is exactly the common neighborhood of the unique minimum-degree layer. This perspective is the key to the exact connectivity proof of Problem 7.2 in \cite{Banerjee2022}.
	
	\begin{example}\label{ex:n210degree}
		Let $n=210=2\cdot 3\cdot 5\cdot 7$. Then Figure \ref{fig:n210model} shows its weighted disjointness model, with each node represents a support class $X_S$, and the number below the class label is its cardinality. Two classes are joined exactly when their index sets are disjoint. The highlighted red node $X_{\{1,2,3\}}$ is the minimum-degree class, and its common neighborhood is the highlighted blue node $X_{\{4\}}$, so every vertex in $X_{\{1,2,3\}}$ has degree $|X_{\{4\}}|=8$.	The degree values for some representative support classes are listed in Table \ref{tab:n210degree}.
	\end{example}
	
	\begin{table}[H]
		\centering
		\caption{Representative degree values for $n=210$.}
		\label{tab:n210degree}
		\begin{tabular}{ccc}
			\toprule
			Support class $X_S$ & $|X_S|$ & Degree of a vertex in $X_S$ \\
			\midrule
			$X_{\{1\}}$ & $48$ & $57$ \\
			$X_{\{2\}}$ & $24$ & $92$ \\
			$X_{\{3\}}$ & $12$ & $120$ \\
			$X_{\{4\}}$ & $8$ & $132$ \\
			$X_{\{1,2\}}$ & $24$ & $22$ \\
			$X_{\{1,2,3\}}$ & $6$ & $8$ \\
			$X_{\{1,2,4\}}$ & $4$ & $12$ \\
			\bottomrule
		\end{tabular}
	\end{table}
	\begin{figure}[H]
		\centering
		\begin{tikzpicture}[
			every node/.style={
				draw,
				ellipse,
				align=center,
				minimum width=12mm,
				minimum height=10mm,
				inner sep=1.2pt
			},
			scale=0.95
			]
			
			\node (s1) at (0,2.8) {$X_{\{1\}}$\\ $\scriptstyle 48$};
			\node (s2) at (2.8,3.7) {$X_{\{2\}}$\\ $\scriptstyle 24$};
			\node (s3) at (5.6,3.6) {$X_{\{3\}}$\\ $\scriptstyle 12$};
			\node[fill=blue!12] (s4) at (8.4,2.5) {$X_{\{4\}}$\\ $\scriptstyle 8$};
			
			\node (p12) at (-1.4,-0.7) {$X_{\{1,2\}}$\\ $\scriptstyle 24$};
			\node (p13) at (-0.8,0.8) {$X_{\{1,3\}}$\\ $\scriptstyle 12$};
			\node (p14) at (3.1,0) {$X_{\{1,4\}}$\\ $\scriptstyle 8$};
			\node (p23) at (5.7,0) {$X_{\{2,3\}}$\\ $\scriptstyle 6$};
			\node (p24) at (9.2,0.9) {$X_{\{2,4\}}$\\ $\scriptstyle 4$};
			\node (p34) at (9.9,-0.7) {$X_{\{3,4\}}$\\ $\scriptstyle 2$};
			
			\node[fill=red!15] (t123) at (8,-1.8) {$X_{\{1,2,3\}}$\\ $\scriptstyle 6$};
			\node (t124) at (3.4,-1.8) {$X_{\{1,2,4\}}$\\ $\scriptstyle 4$};
			\node (t134) at (5.6,-1.8) {$X_{\{1,3,4\}}$\\ $\scriptstyle 2$};
			\node (t234) at (1.1,-1.8) {$X_{\{2,3,4\}}$\\ $\scriptstyle 1$};
			
			\draw (s1)--(s2);
			\draw (s1)--(s3);
			\draw (s1)--(s4);
			\draw (s2)--(s3);
			\draw (s2)--(s4);
			\draw (s3)--(s4);
			
			\draw (s1)--(p23);
			\draw (s1)--(p24);
			\draw (s1)--(p34);
			
			\draw (s2)--(p13);
			\draw (s2)--(p14);
			\draw (s2)--(p34);
			
			\draw (s3)--(p12);
			\draw (s3)--(p14);
			\draw (s3)--(p24);
			
			\draw (s4)--(p12);
			\draw (s4)--(p13);
			\draw (s4)--(p23);
			
			\draw (s1)--(t234);
			\draw (s2)--(t134);
			\draw (s3)--(t124);
			\draw (s4)--(t123);
			
			\draw (p12)--(p34);
			\draw (p13)--(p24);
			\draw (p14)--(p23);
			
		\end{tikzpicture}
		\vspace*{-3mm}
		\caption{Weighted disjointness model for $n=210=2\cdot 3\cdot 5\cdot 7$.}
		\label{fig:n210model}
	\end{figure}
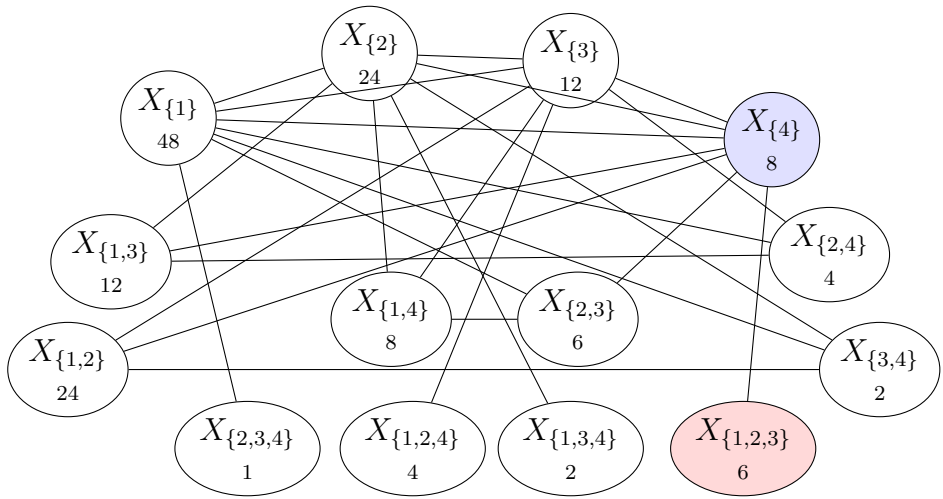

	\section{Exact solution of Problem 7.2}\label{sec:connectivity}
	 This section contains the main contribution of the paper, the exact value of $\kappa(G_2)$. Theorem \ref{thm:mainconnectivity} sharpens Proposition \ref{prop:BanerjeeUpper} from an upper bound to an equality and proves, in addition, that $G_2$ is maximally connected.
	
	\medskip
	 The next lemma recovers Banerjee's separator in the support-set language and shows that it really disconnects the graph.	
	\begin{lemma}\label{lem:separator}
		The class $X_{\{m\}}$ is a vertex cut of $G_2$, and consequently,
		$$
		\kappa(G_2)\leq |X_{\{m\}}|=\prod_{i=1}^{m-1}(p_i-1).
		$$
	\end{lemma}
	
	\begin{proof}
		For any vertex $x\in X_{[m]\setminus\{m\}}$, Corollary \ref{cor:minlayerneighborhood}, gives
		$ N(x)=X_{\{m\}}. $
		Therefore, after deleting the entire class $X_{\{m\}}$, every vertex of $X_{[m]\setminus\{m\}}$ becomes isolated. Hence, $G_2-X_{\{m\}}$ is disconnected, so $X_{\{m\}}$ is a vertex cut. The cardinality formula follows from Proposition \ref{prop:classsize}.
	\end{proof}
	
	\medskip
	The next lemma shows that any smaller deletion leaves at least one vertex in the critical class $X_{\{m\}}$.	
	\begin{lemma}\label{lem:survivingu}
		Let $W\subseteq V(G_2)$ with
		$$
		|W|<\prod_{i=1}^{m-1}(p_i-1).
		$$
		Then there exists a vertex
		$ u\in X_{\{m\}}\setminus W. $
	\end{lemma}
	
	\begin{proof}
		By Proposition \ref{prop:classsize}, we have
		$$
		|X_{\{m\}}|=\prod_{i=1}^{m-1}(p_i-1).
		$$
		Since $|W|<|X_{\{m\}}|$, the set $W$ cannot contain all vertices of $X_{\{m\}}$. Hence some $u\in X_{\{m\}}\setminus W$ survives.
	\end{proof}
	
	\medskip
	 The next lemma shows that every surviving vertex whose zero-set avoids $m$ is directly linked to the surviving anchor $u$.	
	\begin{lemma}\label{lem:directtou}
		Let $u\in X_{\{m\}}$, and let $v\in X_S$ be any vertex with $m\notin S$. Then $u$ and $v$ are adjacent.
	\end{lemma}
	
	\begin{proof}
		Since $m\notin S$, we have $S\cap \{m\}=\emptyset$. Theorem \ref{thm:adjdisjoint} now implies that $v$ is adjacent to every vertex of $X_{\{m\}}$, in particular to $u$.
	\end{proof}
	
	\medskip
	 The next lemma is the key robustness step, every surviving vertex whose zero-set contains $m$ still reaches the anchor through one surviving neighbor.
	\begin{lemma}\label{lem:twostep}
		Let $W\subseteq V(G_2)$ with $ |W|<\prod_{i=1}^{m-1}(p_i-1), $ and choose $u\in X_{\{m\}}\setminus W$ as in Lemma \ref{lem:survivingu}. If $v\in V(G_2)\setminus W$ and $m\in Z(v)$, then $v$ lies in the same connected component as $u$ in $G_2-W$.
	\end{lemma}
	
	\begin{proof}
		Since $v$ survives, it belongs to some class $X_S$ with $m\in S$. So, by Theorem \ref{thm:mindegree}, we have
		$$
		\deg(v)\geq \delta(G_2)=\prod_{i=1}^{m-1}(p_i-1).
		$$
		As $|W|$ is strictly smaller than this number, the deleted set $W$ cannot contain all neighbors of $v$. Hence, there exists a neighbor $ w\in N(v)\setminus W. $ Since $v$ and $w$ are adjacent, Theorem \ref{thm:adjdisjoint} gives
		$ Z(v)\cap Z(w)=\emptyset. $ As $m\in Z(v)$, it follows that $m\notin Z(w)$. Therefore, Lemma \ref{lem:directtou}  implies that $w$ is adjacent to $u$. Hence,  $ v-w-u $ is a path in $G_2-W$, and $v$ lies in the same component as $u$.
	\end{proof}
	
	\medskip
	 The next theorem solves Problem 7.2 of \cite{Banerjee2022} completely. 
	 \begin{theorem}\label{thm:mainconnectivity}
		Let $n=p_1p_2\cdots p_m$ with distinct primes in increasing order. Then
		$$
		\kappa(G_2)=\prod_{i=1}^{m-1}(p_i-1)=\frac{\phi(n)}{p_m-1}.
		$$
	\end{theorem}
	
	\begin{proof}
		Lemma \ref{lem:separator} gives the upper bound
		$$
		\kappa(G_2)\leq \prod_{i=1}^{m-1}(p_i-1).
		$$
		For the reverse inequality, let $W\subseteq V(G_2)$ satisfy
		$$
		|W|<\prod_{i=1}^{m-1}(p_i-1).
		$$
		By Lemma \ref{lem:survivingu}, there exists a surviving vertex $u\in X_{{m}}\setminus W$. Now take any surviving vertex $v\in V(G_2)\setminus W$. If $m\notin Z(v)$, then by Lemma \ref{lem:directtou}, $v$ is adjacent to $u$. If $m\in Z(v)$, then by Lemma \ref{lem:twostep}, $v$ is connected to $u$ in $G_2-W$. Thus, every surviving vertex lies in the same connected component as $u$, so $G_2-W$ is connected. Therefore, no deletion of fewer than $\prod_{i=1}^{m-1}(p_i-1)$ vertices disconnects $G_2$. Combining the upper and lower bounds gives
		$ \kappa(G_2)=\prod_{i=1}^{m-1}(p_i-1). $
		Since $\phi(n)=\prod_{i=1}^m(p_i-1)$, the equivalent form $ \kappa(G_2)=\frac{\phi(n)}{p_m-1} $ follows immediately.
	\end{proof}
	
	\medskip
	The next corollary shows that the graph is maximally connected.	
	\begin{corollary}\label{cor:maxconn}
		The graph $G_2$ satisfies
		$$
		\kappa(G_2)=\delta(G_2)=\prod_{i=1}^{m-1}(p_i-1).
		$$
	\end{corollary}
	
	\begin{proof}
		Theorem \ref{thm:mainconnectivity} gives the value of $\kappa(G_2)$, and Theorem \ref{thm:mindegree} gives the same value for $\delta(G_2)$.
	\end{proof}
	
	\medskip
	 The next corollary identifies the edge connectivity as well.	
	\begin{corollary}\label{cor:edgeconn}
		The edge connectivity of $G_2$ is
		$$
		\lambda(G_2)=\kappa(G_2)=\delta(G_2)=\prod_{i=1}^{m-1}(p_i-1).
		$$
	\end{corollary}
	
	\begin{proof}
		By Proposition \ref{prop:Whitney}, we have
		$$
		\kappa(G_2)\leq \lambda(G_2)\leq \delta(G_2).
		$$
		Corollary \ref{cor:maxconn} shows that the leftmost and rightmost terms are equal. Hence the middle term must equal them as well.
	\end{proof}
	
	 Proposition \ref{prop:BanerjeeUpper} is exactly the estimate left open in \cite{Banerjee2022}. Theorem \ref{thm:mainconnectivity} proves that this estimate is always sharp in the squarefree case and further reveals a stronger property, namely maximal connectivity. 	
	\begin{example}\label{ex:kappa210}
		Let $n=210=2\cdot 3\cdot 5\cdot 7$, see Figure \ref{fig:n210model}. Then
		$ \kappa(G_2)=(2-1)(3-1)(5-1)=8. $ The minimum cut is the class $X_{\{4\}}$, which contains exactly the vertices whose only zero coordinate is in the $7$-component. Deleting this class isolates every vertex of $X_{\{1,2,3\}}$. Table \ref{tab:connectivitycompare} compares several values of the exact formula with the  upper bound (Proposition \ref{prop:BanerjeeUpper}). They coincide in every case, as predicted by Theorem \ref{thm:mainconnectivity}.
	\end{example}
	Recall the case $n=30=2\cdot 3\cdot 5,$ (see Figure \ref{fig:cut30}). The class $X_{\{3\}}$ is a minimum cut. After deleting it, the class $X_{\{1,2\}}$ becomes isolated. Dashed edges are removed together with the cut class.
	\begin{table}[H]
		\centering
		\caption{Comparison of the exact value of $\kappa(G_2)$ with the Banerjee upper bound.}
		\label{tab:connectivitycompare}
		\begin{tabular}{ccccc}
			\toprule
			$n$ & Prime factorization & Banerjee upper bound & Exact $\kappa(G_2)$ & $\delta(G_2)$ \\
			\midrule
			$6$ & $2\cdot 3$ & $1$ & $1$ & $1$ \\
			$30$ & $2\cdot 3\cdot 5$ & $2$ & $2$ & $2$ \\
			$42$ & $2\cdot 3\cdot 7$ & $2$ & $2$ & $2$ \\
			$70$ & $2\cdot 5\cdot 7$ & $4$ & $4$ & $4$ \\
			$210$ & $2\cdot 3\cdot 5\cdot 7$ & $8$ & $8$ & $8$ \\
			$2310$ & $2\cdot 3\cdot 5\cdot 7\cdot 11$ & $48$ & $48$ & $48$ \\
			\bottomrule
		\end{tabular}
	\end{table}
	
	\begin{figure}[H]
		\centering
		\begin{tikzpicture}[
			every node/.style={draw, ellipse, align=center, minimum width=20mm, minimum height=10mm},
			scale=1
			]
			\node[fill=white] (a1) at (0,2.2) {$X_{\{1\}}$};
			\node[fill=white] (a2) at (3,2.9) {$X_{\{2\}}$};
			\node[fill=gray!25] (a3) at (6,2.2) {$X_{\{3\}}$\\ \scriptsize cut};
			\node[fill=red!15] (b12) at (0,0) {$X_{\{1,2\}}$\\ \scriptsize isolated};
			\node[fill=white] (b13) at (3,-0.4) {$X_{\{1,3\}}$};
			\node[fill=white] (b23) at (6,0) {$X_{\{2,3\}}$};
			
			\draw (a1)--(a2);
			\draw (b13)--(a2);
			\draw (b23)--(a1);
			
			\draw[dashed] (a1)--(a3);
			\draw[dashed] (a2)--(a3);
			\draw[dashed] (b12)--(a3);
		\end{tikzpicture}
		\caption{For $n=30$, the class $X_{\{3\}}$ is a minimum cut.}
		\label{fig:cut30}
	\end{figure}
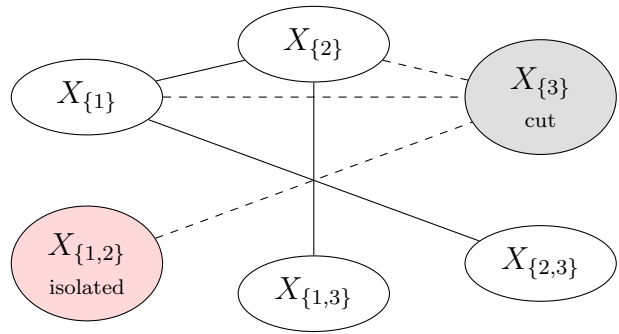
	
	\section{Metric consequences and short routing}\label{sec:distance}
	 The support-set model also yields precise metric information. To the best of our knowledge, the exact distance formula below has not been stated for $G_2$ in the squarefree case. From now on, unless explicitly stated otherwise, we assume $m\geq 3$ whenever metric statements involving diameter $3$ or the center of the graph are discussed.
	
	\medskip
	 The next theorem determines the distance between two vertices entirely from their zero-sets.
	\begin{theorem}\label{thm:distance}
		Let $x\in X_S$ and $y\in X_T$, where $S$ and $T$ are nonempty proper subsets of $[m]$. Then
		$$
		d_{G_2}(x,y)=
		\begin{cases}
			1,& \text{if } S\cap T=\emptyset,\\[2mm]
			2,& \text{if } S\cap T\neq \emptyset \text{ and } S\cup T\neq [m],\\[2mm]
			3,& \text{if } S\cap T\neq \emptyset \text{ and } S\cup T=[m].
		\end{cases}
		$$
	\end{theorem}
	
	\begin{proof}
		If $S\cap T=\emptyset$, then $x$ and $y$ are adjacent by Theorem \ref{thm:adjdisjoint}, so the distance is $1$.
		 Assume now that $S\cap T\neq\emptyset$. Then $x$ and $y$ are not adjacent. If $S\cup T\neq [m]$, choose $k\in [m]\setminus (S\cup T)$. Since $m\geq 3$ is not needed in this case, by Proposition \ref{prop:classsize}, the singleton class $X_{{k}}$ is nonempty. Any vertex $z\in X_{{k}}$ satisfies ${k}\cap S=\emptyset$ and ${k}\cap T=\emptyset$, so by Theorem \ref{thm:adjdisjoint}, $x\sim z$ and $z\sim y$. Hence $d(x,y)\leq 2$. As $x$ and $y$ are not adjacent, so $d(x,y)=2$.
		
		Finally, assume that $S\cap T\neq\emptyset$ and $S\cup T=[m]$. We first show that $d(x,y)\neq 2$. If there were a common neighbor $z\in X_U$, then $U$ would have to be disjoint from both $S$ and $T$. Hence
		$ U\subseteq [m]\setminus (S\cup T)=\emptyset, $ contradicting the fact that $U$ is nonempty. Thus $d(x,y)\geq 3$.
		
		To show that $d(x,y)\leq 3$, note that $S$ and $T$ are proper subsets whose union is all of $[m]$. Therefore, both differences $T\setminus S$ and $S\setminus T$ are nonempty. Choose $a\in T\setminus S$ and $b\in S\setminus T$. Take any $u\in X_{{a}}$ and $v\in X_{{b}}$. Then ${a}\cap S=\emptyset$, so $x\sim u$. Also ${a}\cap{b}=\emptyset$, so $u\sim v$, and finally ${b}\cap T=\emptyset$, so $v\sim y$. Thus
		$ x-u-v-y $ is a path of length $3$, and therefore $d(x,y)=3$.
	\end{proof}
	
	\medskip
	The next corollary gives the diameter of the graph.	
	\begin{corollary}\label{cor:diameter}
		If $m=2$, then $G_2$ is complete bipartite and $\operatorname{diam}(G_2)=2$. If $m\geq 3$, then
		$ \operatorname{diam}(G_2)=3. $
	\end{corollary}
	
	\begin{proof}
		When $m=2$, the only nonempty proper subsets of $[2]$ are ${1}$ and ${2}$, so $G_2$ is the complete bipartite graph with bipartition $X_{{1}}\cup X_{{2}}$, and its diameter is $2$.
		
		For $m\geq 3$, Theorem \ref{thm:distance} shows that every pair of vertices is at distance at most $3$, so $\operatorname{diam}(G_2)\leq 3$. To see that the value $3$ occurs, choose
		$ S=[m]\setminus\{1\},$ and $ T=[m]\setminus\{m\}. $ Then $S$ and $T$ are nonempty proper subsets, $S\cap T\neq\emptyset$, and $S\cup T=[m]$. By Theorem \ref{thm:distance}, any $x\in X_S$ and $y\in X_T$ satisfy $d(x,y)=3$. Hence, $\operatorname{diam}(G_2)=3$.
	\end{proof}
	
	\medskip
	 The next theorem identifies the center of the graph for $m\geq 3$.	
	\begin{theorem}\label{thm:center}
		For $m\geq 3$, the radius of $G_2$ is $2$, and its center is exactly
		$ \mathcal C=\bigcup_{i=1}^m X_{{i}}. $
	\end{theorem}
	
	\begin{proof}
		Let $x\in X_{{i}}$ for some $i\in [m]$. For any vertex $y\in X_T$, if $i\notin T$, then ${i}\cap T=\emptyset$, so $x$ and $y$ are adjacent by Theorem \ref{thm:adjdisjoint}. If $i\in T$, then ${i}\cap T\neq\emptyset$. As $T$ is proper and $m\geq 3$, we cannot have ${i}\cup T=[m]$, indeed, that would force $T=[m]$. Therefore the second case of Theorem \ref{thm:distance} applies, and $d(x,y)=2$. Hence every vertex in a singleton class has eccentricity at most $2$.
		
		Such a vertex is not adjacent to another vertex in the same class by Corollary \ref{cor:falsetwins}, so its eccentricity is exactly $2$. Thus every vertex in $\mathcal C$ is central and the radius is at most $2$. Since the graph is not complete, the radius cannot be $1$, so the radius is exactly $2$.
		
		Now let $x\in X_S$ with $|S|\geq 2$. Choose $a\in S$ and define
		$$
		T=[m]\setminus\{a\}.
		$$
		Because $|S|\geq 2$, the set $S\cap T=S\setminus\{a\}$ is nonempty, and clearly $S\cup T=[m]$. By Theorem \ref{thm:distance}, every $y\in X_T$ satisfies $d(x,y)=3$. Hence the eccentricity of $x$ is at least $3$, so $x$ is not central. Therefore the center is exactly $\mathcal C$.
	\end{proof}
	
	\medskip
	The next corollary gives a uniform routing bound after arbitrary deletions below the connectivity threshold.	
	\begin{corollary}\label{cor:deleteddiam}
		Let $W\subseteq V(G_2)$ satisfy
		$$
		|W|<\kappa(G_2)=\prod_{i=1}^{m-1}(p_i-1).
		$$
		Then $G_2-W$ is connected, every surviving vertex is at distance at most $2$ from some surviving vertex of $X_{{m}}$, and $ \operatorname{diam}(G_2-W)\leq 4. $
	\end{corollary}
	
	\begin{proof}
		Connectivity follows from Theorem \ref{thm:mainconnectivity}. By Lemma \ref{lem:survivingu}, there exists a surviving vertex $u\in X_{{m}}\setminus W$. The proof of Theorem \ref{thm:mainconnectivity}, through Lemmas \ref{lem:directtou} and \ref{lem:twostep}, shows that every surviving vertex lies at distance at most $2$ from $u$ in $G_2-W$. Therefore, any two surviving vertices are at distance at most $4$ from one another.
	\end{proof}
	
	 Neither the exact distance formula nor the resulting diameter and center descriptions are explicit in the existing literature.  They emerge naturally from the support-set description and further demonstrate that the open problem on connectivity is part of a broader metric simplification available only in the squarefree model.
	
	\begin{example}\label{ex:dist210}
		Take $n=210$ and consider the support pairs listed in Table \ref{tab:distanceexamples}. The table illustrates all three cases of Theorem \ref{thm:distance}. In particular, the pair $ S=\{1,2\},$ and $ T=\{2,3,4\} $ has nonempty intersection and full union, so the corresponding vertices are at distance $3$.
	\end{example}
	
	\begin{table}[H]
		\centering
		\caption{Sample distances in $G_2$ for $n=210$.}
		\label{tab:distanceexamples}
		\begin{tabular}{cccc}
			\toprule
			$S$ & $T$ & Relation between $S$ and $T$ & Distance \\
			\midrule
			$\{1\}$ & $\{2\}$ & disjoint & $1$ \\
			$\{1,2\}$ & $\{4\}$ & disjoint & $1$ \\
			$\{1,2\}$ & $\{2,4\}$ & intersecting, union $\neq [4]$ & $2$ \\
			$\{1\}$ & $\{1,3\}$ & intersecting, union $\neq [4]$ & $2$ \\
			$\{1,2\}$ & $\{2,3,4\}$ & intersecting, union $=[4]$ & $3$ \\
			$\{1,3\}$ & $\{2,3,4\}$ & intersecting, union $=[4]$ & $3$ \\
			\bottomrule
		\end{tabular}
	\end{table}
	
	Figure \ref{fig:distancecases} shows the shortest-path patterns from Theorem \ref{thm:distance}. The support-set relation alone determines whether the geodesic length is $1$, $2$, or $3$.
	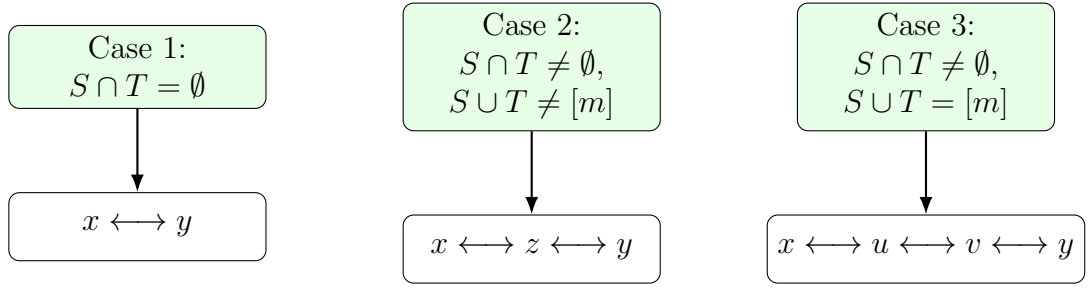
\begin{figure}[H]
		\centering
		\begin{tikzpicture}[
			box/.style={draw, rounded corners, minimum width=34mm, minimum height=9mm, align=center},
			arr/.style={-Latex, thick}
			]
			\node[box, fill=green!10] (c1) at (0,2.8) {Case 1:\\ $S\cap T=\emptyset$};
			\node[box, fill=green!10, right=18mm of c1] (c2) {Case 2:\\ $S\cap T\neq\emptyset$,\\ $S\cup T\neq [m]$};
			\node[box, fill=green!10, right=18mm of c2] (c3) {Case 3:\\ $S\cap T\neq\emptyset$,\\ $S\cup T=[m]$};
			
			\node[box, below=11mm of c1] (p1) {$x\longleftrightarrow y$};
			\node[box, below=11mm of c2] (p2) {$x\longleftrightarrow z\longleftrightarrow y$};
			\node[box, below=11mm of c3] (p3) {$x\longleftrightarrow u\longleftrightarrow v\longleftrightarrow y$};
			
			\draw[arr] (c1)--(p1);
			\draw[arr] (c2)--(p2);
			\draw[arr] (c3)--(p3);
		\end{tikzpicture}
		\vspace*{-2mm}
		\caption{Shortest-path patterns from Theorem \ref{thm:distance}.}
		\label{fig:distancecases}
	\end{figure}
	
	\section{Computational realization, algorithm, and comparative data}\label{sec:algorithmic}
	 The exact formula for $\kappa(G_2)$ immediately leads to an efficient computation procedure and reveals an unexpected arithmetic asymmetry, once the prime factors are ordered, the size of the minimum cut depends only on the first $m-1$ primes.
	
	\medskip
	The next proposition rewrites the exact formula of $\kappa(G_2)$ in Euler form.	
	\begin{proposition}\label{prop:eulerform}
		Let $n=p_1p_2\cdots p_m$ with distinct primes in increasing order. Then
		$$
		\kappa(G_2)=\frac{\phi(n)}{p_m-1}.
		$$
	\end{proposition}

	\medskip
	The next corollary shows that changing only the largest prime does not change the connectivity value.	
	\begin{corollary}\label{cor:largestinvariance}
		Fix distinct primes $p_1<\cdots<p_{m-1}$ and let $q>p_{m-1}$ be any prime. If $n'=p_1p_2\cdots p_{m-1}q$, then
		$$
		\kappa\big(G_2(n')\big)=\prod_{i=1}^{m-1}(p_i-1).
		$$
		Hence, for fixed first $m-1$ primes, the size of the minimum cut is independent of the final prime.
	\end{corollary}
	
	\begin{proof}
		Apply Theorem \ref{thm:mainconnectivity} to $n'$. As the ordered prime list for $n'$ is exactly $p_1,\dots,p_{m-1},q$, the product defining $\kappa\big(G_2(n')\big)$ involves only the first $m-1$ primes.
	\end{proof}
	
	\medskip
	 The following algorithm computes $\kappa(G_2)$ directly from the ordered prime factors.	
	\begin{algorithm}[H]
		\caption{Computation of $\kappa(G_2)$ for squarefree $n=p_1p_2\cdots p_m$}
		\label{alg:kappa}
		\begin{algorithmic}[1]
			\Require Distinct primes $p_1<\cdots<p_m$
			\Ensure The value of $\kappa(G_2)$
			\If{$m=1$}
			\State \Return $0$
			\ElsIf{$m=2$}
			\State \Return $p_1-1$
			\Else
			\State $K\gets 1$
			\For{$i=1$ to $m-1$}
			\State $K\gets K\cdot (p_i-1)$
			\EndFor
			\State \Return $K$
			\EndIf
		\end{algorithmic}
	\end{algorithm}
	
	\medskip
	The following theorem validates Algorithm \ref{alg:kappa}.	
	\begin{theorem}\label{thm:algorithmcorrect}
		Algorithm \ref{alg:kappa} returns the exact value of $\kappa(G_2)$ for every squarefree modulus $n$ with ordered prime factors $p_1<\cdots<p_m$.
	\end{theorem}
	
	\begin{proof}
		If $m=1$, then $n$ is prime, every nonzero element is a unit, and $G_2$ is empty. So, returning $0$ is correct. If $m=2$, Theorem \ref{thm:mainconnectivity} implies
		$ \kappa(G_2)=p_1-1, $ so the second branch is correct. If $m\geq 3$, Theorem \ref{thm:mainconnectivity} yields
		$$
		\kappa(G_2)=\prod_{i=1}^{m-1}(p_i-1),
		$$
		which is exactly the value accumulated in the loop. Hence, every branch returns the correct value.
	\end{proof}
	
	\medskip
	The next proposition records the computational cost of the algorithm after factorization is known.	
	\begin{proposition}\label{prop:complexity}
		Once the distinct prime factors of $n$ are known and ordered, Algorithm \ref{alg:kappa} computes $\kappa(G_2)$ using $m-1$ multiplications. In particular, its arithmetic complexity is $O(m)$.
	\end{proposition}
	
	\begin{proof}
		For $m\geq 3$, the loop executes exactly once for each of the first $m-1$ prime factors, and each iteration performs one multiplication. The cases $m=1$ and $m=2$ are constant-time branches. Therefore, the arithmetic cost is linear in the number of prime factors.
	\end{proof}
	
	\medskip
	The next corollary describes how the connectivity changes when a new largest prime factor is appended.	
	\begin{corollary}\label{cor:appendprime}
		Let $n=p_1p_2\cdots p_m$ with $m\geq 2$, and let $q>p_m$ be a prime. Then
		$$
		\kappa\big(G_2(nq)\big)=(p_m-1)\kappa\big(G_2(n)\big).
		$$
	\end{corollary}
	
	\begin{proof}
		Applying Theorem \ref{thm:mainconnectivity} to $nq$ gives the required identity
		$$
		\kappa\big(G_2(nq)\big)=\prod_{i=1}^{m}(p_i-1)=(p_m-1)\prod_{i=1}^{m-1}(p_i-1)=(p_m-1)\kappa\big(G_2(n)\big).
		$$
	\end{proof}
	
	 The earlier literature provides structural and spectral decompositions of $\Gamma(\mathbb Z_n)$, but not a direct computation rule for $\kappa(G_2)$. Theorem \ref{thm:algorithmcorrect} shows that once the squarefree factorization is available, the exact connectivity can be read off with linear complexity and no graph construction at all.
	
	\begin{example}\label{ex:algorithm}
		For $n=2310=2\cdot 3\cdot 5\cdot 7\cdot 11$, Algorithm \ref{alg:kappa} returns
		$$
		(2-1)(3-1)(5-1)(7-1)=48.
		$$
		The same value is obtained from Proposition \ref{prop:eulerform}, since
		$ \phi(2310)=480,$ so $\tfrac{\phi(2310)}{11-1}=48. $
		This agrees with Table \ref{tab:algorithmicdata}.
	\end{example}
	
	\begin{table}[H]
		\centering
		\caption{Comparative computational data for the exact formula.}
		\label{tab:algorithmicdata}
		\begin{tabular}{cccccc}
			\toprule
			$n$ & Ordered primes & $\phi(n)$ & $\tfrac{\phi(n)}{p_m-1}$ & Exact $\kappa(G_2)$ & Diameter \\
			\midrule
			$6$ & $(2,3)$ & $2$ & $1$ & $1$ & $2$ \\
			$30$ & $(2,3,5)$ & $8$ & $2$ & $2$ & $3$ \\
			$42$ & $(2,3,7)$ & $12$ & $2$ & $2$ & $3$ \\
			$70$ & $(2,5,7)$ & $24$ & $4$ & $4$ & $3$ \\
			$210$ & $(2,3,5,7)$ & $48$ & $8$ & $8$ & $3$ \\
			$2310$ & $(2,3,5,7,11)$ & $480$ & $48$ & $48$ & $3$ \\
			\bottomrule
		\end{tabular}
	\end{table}
	Figure \ref{fig:flowchart} shows the flowchart of Algorithm \ref{alg:kappa}. The graph-theoretic computation is reduced to a short arithmetic product once the squarefree prime factorization is available.
	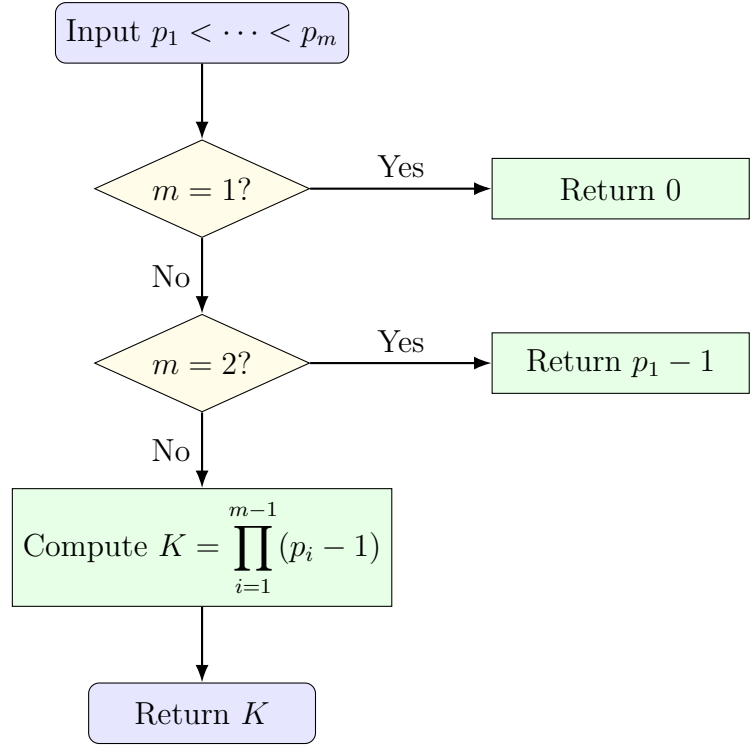
\begin{figure}[H]
		\centering
		\begin{tikzpicture}[
			node distance=10mm and 18mm,
			startstop/.style={draw, rounded corners, fill=blue!10, minimum width=30mm, minimum height=8mm, align=center},
			process/.style={draw, rectangle, fill=green!10, minimum width=34mm, minimum height=8mm, align=center},
			decision/.style={draw, diamond, aspect=2.2, fill=yellow!10, align=center},
			arrow/.style={-Latex, thick}
			]
			\node[startstop] (start) {Input $p_1<\cdots<p_m$};
			\node[decision, below=of start] (d1) {$m=1$?};
			\node[decision, below=of d1] (d2) {$m=2$?};
			\node[process, right=24mm of d1] (r0) {Return $0$};
			\node[process, right=24mm of d2] (r1) {Return $p_1-1$};
			\node[process, below=of d2] (loop) {Compute $K=\prod_{i=1}^{m-1}(p_i-1)$};
			\node[startstop, below=of loop] (stop) {Return $K$};
			
			\draw[arrow] (start) -- (d1);
			\draw[arrow] (d1) -- node[left] {No} (d2);
			\draw[arrow] (d1) -- node[above] {Yes} (r0);
			\draw[arrow] (d2) -- node[above] {Yes} (r1);
			\draw[arrow] (d2) -- node[left] {No} (loop);
			\draw[arrow] (loop) -- (stop);
		\end{tikzpicture}
		\vspace*{-2mm}
		\caption{Flowchart for Algorithm \ref{alg:kappa}.}
		\label{fig:flowchart}
	\end{figure}
	
	\section{Conclusion and future work}\label{sec:conclusion}
	
	We have given a complete solution to Problem 7.2 of Banerjee \cite{Banerjee2022}. For the squarefree modulus $ n=p_1p_2\cdots p_m,$ with $2\leq p_1<\cdots<p_m, $ the induced subgraph $G_2$ of the comaximal graph $\Gamma(\mathbb Z_n)$ on the nonzero nonunits satisfies
	$$
	\kappa(G_2)=\prod_{i=1}^{m-1}(p_i-1)=\frac{\phi(n)}{p_m-1}.
	$$
	The proof is based on a support-set model obtained from the Chinese remainder theorem. In this model, vertices are indexed by nonempty proper subsets of $[m]$, adjacency becomes set disjointness, and $G_2$ is a weighted blow-up of a disjointness graph. This viewpoint also yields explicit class sizes, degree formulas, the unique minimum-degree layer, a transparent minimum separator, maximal connectivity, the exact edge connectivity, and a metric description including the diameter and center.
	
	The main limitation of the paper is that the argument is tailored to the squarefree case. Once repeated prime powers are allowed, coordinates no longer reduce to simple zero versus nonzero choices, and the disjointness model must be replaced by a more delicate multilevel description. That nonsquarefree situation seems genuinely harder and cannot be treated by a superficial modification of the present proof details of results. A second limitation is that we determined the exact value of the connectivity but did not classify all minimum vertex cuts. SO, we leave the following two problems
	
	\begin{problem}
		Classify all minimum vertex cuts of $G_2$ when $n=p_1p_2\cdots p_m$ is squarefree. Is every minimum separator forced, up to graph automorphism, to arise from the class $X_{{m}}$?
	\end{problem}
	
	\begin{problem}
		Extend Theorem \ref{thm:mainconnectivity} to general moduli $n$ with repeated prime powers. In particular, determine $\kappa(G_2)$ for $n=p_1^{\alpha_1}\cdots p_t^{\alpha_t}$ with at least one $\alpha_i>1$.
	\end{problem}
	
	\section*{Declarations}
	\noindent \textbf{Data Availability:} There is no data associated with this article.
	
	\noindent \textbf{Funding:} The authors did not receive support from any organization for the submitted work.
	
	\noindent \textbf{Conflict of interest:} The authors have no competing interests to declare that are relevant to the content of this article.
	
	\noindent\textbf{Note:} 
	I welcome any comments and suggestions regarding this article; please feel free to contact
	me at \texttt{\href{mailto:bilalahmadrr@gmail.com}{bilalahmadrr@gmail.com}}.

\end{document}